\documentclass{amsart}
\usepackage[utf8]{inputenc} 
\usepackage[T1]{fontenc} 
\usepackage[english]{babel}
\usepackage{comment}
\usepackage{tikz} 
\usetikzlibrary{decorations.markings}
\usepackage{pgfplots} 
\pgfplotsset{compat=1.17}
\usepgfplotslibrary{fillbetween}
\usepackage{mathrsfs}
\usepackage{color}
\newcommand{\CB}[1]{{\color{blue}#1}}
\newcommand{\CR}[1]{{\color{red}#1}}
\newcommand{\CG}[1]{{\color{green}#1}}

\usepackage{tikz} 
\usepackage{pgfplots} 
\usepackage{graphicx}
\usepackage{comment}
\pgfplotsset{compat=1.16}

\usepackage{tikz-cd} 
\tikzcdset{arrow style=tikz, arrows={thick}, diagrams={>=stealth}}

\theoremstyle{plain} 
\newtheorem{theorem}{Theorem}
\newtheorem{lemma}[theorem]{Lemma}

\newtheorem*{theorem*}{Theorem}

\theoremstyle{remark}

\newtheorem*{conjecture}{Conjecture}

\DeclareMathOperator{\Res}{Res}

\author{Sarah May Instanes}
\address{Department of Mathematical Sciences, Norwegian University of Science and Technology (NTNU), NO-7491 Trondheim, Norway} 
\email{sarah.m.instanes@ntnu.no}
\date{\today}
\title{Point evaluation for polynomials on the circle}\begin{document} 
	\begin{abstract}
		We study the constant $\mathscr{C}_{d,p}$ defined as the smallest constant $C$ such that $\|P\|_\infty^p \leq C\|P\|_p^p$ holds for every polynomial $P$ of degree $d$, where we consider the $L^p$ norm on the unit circle. We conjecture that $\mathscr{C}_{d,p} \leq dp/2+1$ for all $p \geq 2$ and all degrees $d$. We show that the conjecture holds for all $p \geq 2$ when $d \leq 4$ and for all $d$ when $p \geq 6.8$. 
	\end{abstract}
	\maketitle
	\section{introduction}\label{sec:intro}
	Let $\mathbb{T}$ denote the unit circle and let $L^p(\mathbb{T})$ be the set of all functions on $\mathbb{T}$ for which the $L^p$ norm 
	\[\|f\|_p=\left(\int_{-\pi}^{\pi} |f(e^{i \theta})|^p \frac{d \theta}{2 \pi}\right)^{\frac{1}{p}}\] is finite. In this paper we consider the norm of point evaluation at $z=1$ for polynomials of degree at most $d$ on the unit circle under the $L^p$ norm. More precisely we study the constant $\mathscr{C}_{d,p}$ defined as the smallest $C>0$ such that
	\[\|P\|_\infty^p \leq C \|P\|_p^p,\]
	for any polynomial of degree at most $d$. This is a classical problem related to Nikolskii-type inequalities. The case $p=2$ is particularly favourable. It follows from the Cauchy--Schwarz inequality and Parseval's identity that $\mathscr{C}_{d,2}=d+1$. The classical power trick, see e.g.\ Timan \cite[page 229]{Timan}, tells us that for any integer $n$ it holds that $\mathscr{C}_{d,np} \leq \mathscr{C}_{nd,p}$, and consequently it follows from Hölder's inequality that $\mathscr{C}_{d,p} \leq d \lceil p/2\rceil+1$. We propose the following stronger conjecture.

	\begin{conjecture}
		If $d \geq 1$ and $p \geq 2$, then $\mathscr{C}_{d,p} \leq dp/2+1$.
	\end{conjecture}
	Again by applying the power trick, this conjecture holds for all $p \geq 2$ if verified for $2 \leq p \leq 4$ and all degrees $d$. We also conjecture that $\mathscr{C}_{d,p}/(dp/2+1)$ is decreasing in $p$. Our main result reads as follows.
	\begin{theorem}\label{thm:main}
		If $1 \leq d \leq 4$ and $2 \leq p \leq 4$, then 
		\[\mathscr{C}_{d,p} \leq \frac{dp}{2}+1.\]
	\end{theorem}	
	In fact we improve the above bound slightly, as stated in Theorem \ref{thm:uppergood} and illustrated in Figure \ref{fig:upd234}.
	The problem of determining $\mathscr{C}_{d,p}$ can be equivalently formulated as the extremal problem 
	\begin{equation}\label{eq:extprob}
		\frac{1}{\mathscr{C}_{d,p}} = \inf_{P \in \mathscr{P}_d} \{ \|P\|_p^p : P(1)=1\},
	\end{equation}
	where $\mathscr{P}_d$ consists of all polynomials of degree at most $d$. 
	For $1<p<\infty$, the extremal problem \eqref{eq:extprob} has a unique solution $\varphi_{d,p}$. The extremal function $\varphi_{d, p}$ is of degree $d$ and has $d$ zeroes on the unit circle $\mathbb{T}$ and $P(e^{i \theta})=\overline{P(e^{-i \theta})}$. 
	\begin{figure}
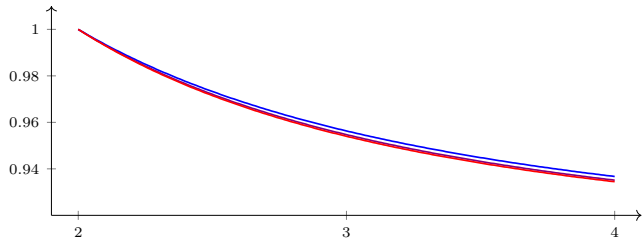

		\centering
		\begin{tikzpicture}[scale=0.8]
			\begin{axis}[
				width=0.9\linewidth,
				height=0.4\linewidth,
				axis lines=middle,
				xmin=1.9, xmax=4.1,
				ymin=0.92, ymax=1.01,
				xtick={2,3,4},
				xticklabels={$\scriptstyle 2$,$\scriptstyle 3$,$\scriptstyle 4$},
				ytick={0.86, 0.88, 0.9, 0.92, 0.94, 0.96, 0.98, 1},
				yticklabels={$\scriptstyle 0.86$, $\scriptstyle 0.88$, $\scriptstyle 	0.9$, $\scriptstyle 0.92$,$\scriptstyle 0.94$, $\scriptstyle 0.96$,$\scriptstyle 0.98$, $\scriptstyle 1$},
				every axis x label/.style={at={(ticklabel* cs:0.1)}, anchor=west,},
				every axis y label/.style={at={(ticklabel* cs:0.1)}, anchor=south,},
				axis line style={->}
				]
				\input{up_d_2}
				\input{up_d_3}
				\input{up_d_4}
			\end{axis}
		\end{tikzpicture}
		\caption{The upper bound for $\mathscr{C}_{d,p}/(dp/2+1)$ for \CB{$d=2$}, \CG{$d=3$} and \CR{$d=4$} from Theorem \ref{thm:uppergood}.}
		\label{fig:upd234}
	\end{figure}
	Brevig, Chirre, Ortega-Cerdà and Seip \cite{Brevig} have studied the related problem of determining $\mathscr{C}_p$ defined as the smallest constant $C$ such that $|f(0)|^p \leq C\|f\|_p^p$
	for all functions in the Paley--Wiener space $PW^p$. Levin and Lubinsky \cite{LevinLubinski} have shown that these two problems are connected via the limit 
	\[\lim_{d \to \infty} \frac{\mathscr{C}_{d,p}}{d}=\mathscr{C}_p.\]
	Our conjecture implies the result $\mathscr{C}_p \leq p/2$ from \cite{Brevig}. Similarly to what was done in \cite{Brevig}
	we provide an extended power trick (see Theorem \ref{thm:24analogi}) and use information on the zeroes of the extremal function in order to prove Theorem \ref{thm:main}. The obstacle to further progress is the limited information on the zeroes. In \cite{Brevig} information about the zeroes is obtained by considering prolate spheroidal wave functions.
	
	For fixed values of $d$ and $p$ we can obtain a good lower bound for $\mathscr{C}_{d,p}$ by numerically optimizing over the zeroes of the polynomial. The lower bound  obtained numerically for $d=4$ and $2 \leq p \leq 4$ is seen in Figure \ref{fig:d4}, together with the corresponding upper bound. In Section \ref{sec:num} we also show that a certain specific polynomial provides a lower bound very close to that obtained by numerical optimization. 
	
	We also consider upper bounds for $\mathscr{C}_{d,p}$ for larger values of $p$ and prove the following result. In this paper $B(x,y)$ denotes the beta function. 
	\begin{theorem}\label{thm:plarge}
		If $d \geq 1$ and $p \geq 1$, then 
		\[\mathscr{C}_{d,p} \leq \frac{d\pi}{B((p+1)/2,1/2)}.\]
	\end{theorem}	
	Due to well-known asymptotics for the beta function it follows that 
	\[\mathscr{C}_{d,p} \leq d\left( \sqrt{\frac{\pi p}{2}}+O\left(\frac{1}{\sqrt{p}}\right) \right),\]
	as $p \to \infty$. In particular if $p \geq 6.8$ then the conjectured upper bound  $\mathscr{C}_{d,p} \leq dp/2+1$ holds for all degrees $d$.
	
	This paper is organized as follows. We start with some preliminary results 	in Section 2. Next, in Section~\ref{sec:Genpower} ~we prove a generalized power trick, which provides the foundations for proving Theorem \ref{thm:main}. In Section \ref{sec:tn} we give two results on the zeroes of the extremal functions in \eqref{eq:extprob}. Section \ref{sec:mainproof} contains the proof of Theorem \ref{thm:main}. In Section \ref{sec:plarge} we prove Theorem \ref{thm:plarge}. Finally in Section \ref{sec:num} we provide some numerical estimates for the zeroes of $\varphi_{d, p}$ and corresponding lower bounds for $\mathscr{C}_{d,p}$. 
	\begin{figure}
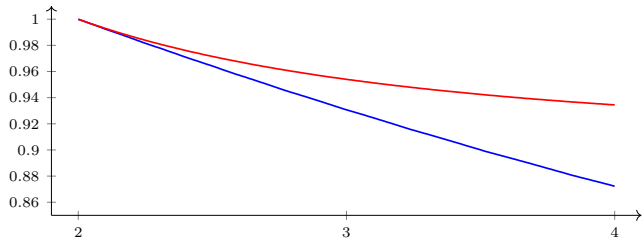

		\centering
		\begin{tikzpicture}[scale=0.8]
			\begin{axis}[
				width=0.9\linewidth,
				height=0.4\linewidth,
				axis lines=middle,
				xmin=1.9, xmax=4.1,
				ymin=0.85, ymax=1.01,
				xtick={2,3,4},
				xticklabels={$\scriptstyle 2$,$\scriptstyle 3$,$\scriptstyle 4$},
				ytick={0.86, 0.88, 0.9, 0.92, 0.94, 0.96, 0.98, 1},
				yticklabels={$\scriptstyle 0.86$, $\scriptstyle 0.88$, $\scriptstyle 	0.9$, $\scriptstyle 0.92$,$\scriptstyle 0.94$, $\scriptstyle 0.96$,$\scriptstyle 0.98$, $\scriptstyle 1$},
				every axis x label/.style={at={(ticklabel* cs:0.1)}, anchor=west,},
				every axis y label/.style={at={(ticklabel* cs:0.1)}, anchor=south,},
				axis line style={->}
				];
				\input{low_d_4_optimizer}
				\input{up_d_4}
			\end{axis}
		\end{tikzpicture}
		\caption{The \CB{lower} and \CR{upper} bound for $\mathscr{C}_{4,p}/(2p+1)$. The \CB{lower} bound is obtained by numerical optimization, while the \CR{upper} bound is from Theorem \ref{thm:uppergood}. }
		\label{fig:d4}
	\end{figure}
	\section{Preliminaries}\label{sec:prelims}
	We start with some preliminary results on the existence and structure of the solution of \eqref{eq:extprob}.
	
	\begin{lemma}
		Let $1<p<\infty$ and $d \geq 1$. Then there exists a unique solution of \eqref{eq:extprob}.
	\end{lemma}
	
	\begin{proof}
		Let $(P_n)_n$ be a sequence in $\mathscr{P}_d$ such that $P_n(1)=1$ and $(\|P_n\|_p^p)_n$ converges to $1/\mathscr{C}_{d,p}$. Since the space $\mathscr{P}_d$ is a finite dimensional vector space there exists a subsequence of polynomials $(P_{n_k})_{n_k}$ that converges to a polynomial $P$ in $\mathscr{P}_d$ such that $\|P\|_p^p=1/\mathscr{C}_{d,p}$, and so a solution exists. Uniqueness follows from strict convexity of $L^p(\mathbb{T})$.
	\end{proof}

	As the extremal function $\varphi_{d, p}$ is unique, we can also deduce its structure. 
	
	\begin{lemma}\label{lem:structure}
		Let $1<p<\infty$. Let $d \geq 1$, and let $k=\lfloor d/2 \rfloor$. Then the unique solution of \eqref{eq:extprob} is of the form
		\[\varphi_{d, p}(z)= C\prod_{j=1}^{k} (z-e^{i t_j})(z-e^{-i t_j}),\]
		for $d$ even and 
		\[\varphi_{d, p}(z)=C(z+1) \prod_{j=1}^{k} (z-e^{i t_j})(z-e^{-i t_j}),\]
		for $d$ odd, where $0 < t_1 \leq t_2 \leq \dots \leq t_{k} \leq \pi$ and $C$ is a normalization constant ensuring that $\varphi_{d,p}(1)=1$.
	\end{lemma}
	
	\begin{proof}
		Assume $P$ is a solution of \eqref{eq:extprob} of degree $n<d$ and let $Q(z)=zP(z)$ be a polynomial of degree at most $d$. Then $\|Q\|_p=\|P\|_p$ and $Q(1)=P(1)=1$. However, since the extremal function is unique this is not possible, so the extremal function must have degree $d$.
		
		Next, let $P(z)=a_0+a_1 z + \dots + a_d z^d$ be a solution of \eqref{eq:extprob}. If $a_0=0$ then the polynomial $P(z)/z$ is also a solution of \eqref{eq:extprob}, and so we must have $a_0 \neq 0$. Define the polynomial $Q$ by $Q(z)=z^d \overline{P(1/\overline{z})}=\overline{a_d}+ \overline{a_{d-1}}z + \dots + \overline{a_0}z^d$. Then $Q$ is of degree $d$ with $Q(1)=P(1)=1$ and $\|Q\|_p=\|P\|_p$. Since the extremal function is unique it follows that $Q=P$. In particular if $re^{i t}$ (with $r>0$) is a zero of $P$, then so is $(1/r) e^{i t}$. 
		
		We now show that the zeroes of the extremal function must all be on the unit circle. Assume to the contrary that $re^{i t}$ is a zero of the extremal function $P$ and that $r \neq 1$, so that $(1/r) e^{i t}$ also is a zero of $P$. We will construct a polynomial $R$ of degree $d$ to show that $P$ cannot be extremal. To do so we let \[A(\theta)=(e^{i \theta}-re^{it})(e^{i \theta}-(1/r)e^{it}),\]
		and \[B(\theta)=(e^{i \theta}-1)^2.\]
		Assume first that $0<t<\pi$. Pick $\varepsilon>0$ so small that 
		\[|A(\theta)| > |A(\theta)-i\varepsilon B(\theta)|>0, \]
		for any $0<\theta<2 \pi$. To see that such an epsilon exists we first let \[m=\min_{\theta \in [0, 2\pi]} |A(e^{i\theta})|\] and observe that $m>0$ since $r \neq 1$. Using the reverse triangle inequality we see that for any $0<\varepsilon<m/4$ it follows that $|A(\theta)-i\varepsilon B(\theta)|>0$. Further, to show that 
		\begin{equation}\label{eq:eps2}
			|A(\theta)| > |A(\theta)-i\varepsilon B(\theta)|,
		\end{equation}
		we see that it suffices to show
		\[i(\overline{A(\theta)}B(\theta)-A(\theta)\overline{B(\theta)})-\varepsilon|B(\theta)|^2>0.\]  
		We see that
		\[\begin{split}
			&i(\overline{A(\theta)}B(\theta)-A(\theta)\overline{B(\theta)})-\varepsilon|B(\theta)|^2 \\ =&\frac{8 \sin(t)((1-r)^2+2r(1-\cos(\theta-t))) \sin^2(\theta/2)}{r}-\varepsilon16 \sin^4(\theta/2).
		\end{split}\]
		In particular, if we choose $\varepsilon<\sin(t)(1-r)^2/(2r)$ then \eqref{eq:eps2} holds. We now define the polynomial 
		\[R(z)=\frac{(z-re^{it})(z-(1/r)e^{it})-i\varepsilon (z-1)^2}{(z-re^{it})(z-(1/r)e^{it})} P(z),\] of degree $d$. Then $|R(z)|<|P(z)|$ for all $z \neq 1$ on the unit circle and $\|R\|_p < \|P\|_p$. It follows that $P$ cannot be an extremal function. This means that we have shown that if $re^{it}$ is a zero of the extremal function and $0<t<\pi$ then $r=1$. The arguments for the cases $t=\pi$ and $\pi<t<2\pi$ are similar. In the case $\pi<t<2\pi$ we can pick $\varepsilon>0$ so small that \[|A(\theta)| > |A(\theta)+i\varepsilon B(\theta)|>0, \]
		for any $0<\theta<2 \pi$ and adjust the definition of $R$ accordingly to show that $P$ cannot be extremal. For $t=\pi$, in fact for any $\pi/2 < t < 3\pi/2$, we can pick 
		$\varepsilon>0$ so small that \[|A(\theta)| > |A(\theta)+\varepsilon B(\theta)|>0, \]
		and again adjust the definition of $R$ accordingly to show that $P$ cannot be extremal. This means that all zeroes of the extremal function must have absolute value $1$. 
		
		Finally by considering the polynomial $S(z)=\overline{P(\overline{z})}$ it follows that if $P$ is extremal then so is $S$, and hence $P=S$. In particular, if $(z-e^{it_j})$ is a factor of $P$ then so is $(z-e^{-it_j})$. This implies that either $(z-e^{it_j})(z-e^{-it_j})$ is a factor of $P$ or $(z-e^{it_j})=(z-e^{-it_j})=(z+1)$ is  a factor of $P$. In particular, if $d$ is odd, $(z+1)$ must be a factor of $P$. The claimed structure follows.
	\end{proof}
	Note that when $d$ is even, then by a shift the problem \eqref{eq:extprob} corresponds to considering real trigonometric polynomials. Given the structure from Lemma \ref{lem:structure}, the case $d=1$ is easily solved. We see that $\varphi_{1,p}(z)=(1+z)/2$, and calculating the norm it follows that $\mathscr{C}_{1,p}=\pi/(B((p+1)/2, 1/2))$. Next we consider the classical power trick.
	
	\begin{lemma}\label{lem:PowerAlt2}
		Let $p \geq 1$ and let $n$ and $d$ be positive integers. It then follows that 
		\[\mathscr{C}_{d,np} \leq \mathscr{C}_{nd,p}.\] 
	\end{lemma}
	
	\begin{proof}
		Let $f$ be a polynomial of degree $d$ and let $g$ be a polynomial of degree $nd$ defined by $g(z)=(f(z))^n$. Then we have
		\[|f(1)|^{np}=|g(1)|^p \leq \mathscr{C}_{nd,p} \|g\|^p_{p}=\mathscr{C}_{nd,p} \|f\|^{np}_{np},\]
		which yields that $\mathscr{C}_{d,np} \leq \mathscr{C}_{nd,p}$.
	\end{proof}
	We observe that if we can show $\mathscr{C}_{d,p} \leq dp/2+1$ for $2 \leq p <4$ and all $d \geq 2$, then this bound can be extended to all $p \geq 2$ by Lemma \ref{lem:PowerAlt2}. Since this is currently out of reach we instead observe the following. 
	\begin{lemma}\label{lem:upper_power}
		If $2<p<\infty$ and $d \geq 1$ then $\mathscr{C}_{d,p} \leq d\lceil p/2\rceil+1$.
	\end{lemma}
	
	\begin{proof}
		Fix an integer $d$ and let $n$ be such that $2n<p \leq 2(n+1)$, and thus $\lceil p/2 \rceil=n+1$. By Lemma \ref{lem:PowerAlt2} it follows that $\mathscr{C}_{d,2(n+1)} \leq \mathscr{C}_{(n+1)d, 2} = (n+1)d+1$. Then using Hölder's inequality it follows that 
		\[\|P\|_\infty^{2(n+1)} \leq ((n+1)d+1)\|P\|_{2(n+1)}^{2(n+1)} \leq  ((n+1)d+1)\|P\|_{p}^{p} \|P\|_\infty^{2(n+1)-p}.\]
		That is $\|P\|_\infty^{p} \leq ((n+1)d+1)\|P\|_{p}^{p},$ which concludes the proof.
	\end{proof}
	
	\section{Extending the power trick}\label{sec:Genpower}
	In this section we extend the power trick by developing a representation formula, which will later be used to determine an upper bound for $\mathscr{C}_{d,p}$. This  representation formula is an analogue of \cite[Theorem 2.4]{Brevig}. We note that in the proof below we exploit the symmetry of the zeroes, leading to a slightly less general representation formula, but somewhat simpler proof than in \cite[Theorem 2.4]{Brevig}.
	
	\begin{theorem}\label{thm:24analogi}
		Let $P$ be a polynomial of degree $d \geq 1$ with $d$ zeroes on the unit circle. Let $k=\lfloor d/2 \rfloor$. Assume $P(e^{i \theta})=\overline{P(e^{-i\theta})}$. Denote the zeroes of $P$ by $e^{i t_1}$, $e^{i t_2}, \dots, e^{i t_{k}}$ and $e^{-i t_1}$, $e^{-i t_2}, \dots e^{-i t_{k}}$, where $0<t_1 \leq t_2 \leq \dots, \leq t_{k} \leq \pi$. If $d$ is odd then $e^{-i t_{k+1}}=-1$ is also a zero and $t_{k+1}=\pi$. For notation purposes we let $t_0=0$, and also for $d$ even we let $t_{k+1}=\pi$. Let $0<q<\infty$. Then 
		\[|P(1)|^q=\sum_{n=0}^{k} \int_{t_n}^{t_{n+1}}|P(e^{i \theta})|^q \frac{\sin((\frac{d}{2}q+\frac{1}{2}) \theta-\pi q n)}{\sin\frac{\theta}{2}} \, \frac{d \theta}{\pi}.\]
	\end{theorem}
	
	\begin{proof}
		By rescaling we assume $P(1)=1$. First, let $d$ be even. Then 
		\[P(z)=C\prod_{j=1}^{k}(z-e^{it_j})(z-e^{-it_j}),\] where $C>0$ is a normalizing constant. We fix a number $q>0$ and define $f_q(z)$ on  $\mathbb{C}\setminus \{e^{i x} : t_1/2 \leq x \leq 2\pi-t_1/2\}$ as the function
		\[f_q(z)=C^q\prod_{j=1}^{k}((z-e^{it_j})(z-e^{-it_j}))^q,\]
		and observe that $f_q(1)=1$.
		
		Let $\varepsilon$ and $\delta$ be two small positive numbers. Define $G_{\varepsilon, \delta}$ to be the closed curve oriented counter-clockwise and connecting the points 
		\[ e^{i \delta}, \quad e^{i t_1/2}, \quad (1-\varepsilon) e^{i t_1/2}, \quad (1-\varepsilon) e^{-i t_1/2} \quad e^{-i t_1/2}, \quad e^{-i \delta}, \] where, when traversing from $e^{-i \delta}$ to $e^{i \delta}$ the contour follows the circle arc centred at $1$ and passing trough $e^{-i \delta}$ and $e^{i \delta}$. We denote this arc by $C_\delta$, and denote the remaining part $G_{\varepsilon, \delta} \setminus C_\delta$ by $\Gamma_{ \varepsilon, \delta}$. See Figure \ref{Fig:thm24Analogi} for an illustration. We observe that $G_{\varepsilon, \delta}$ is a closed simple curve and that the function $z \mapsto f_q(z)/(z-1)$ is analytic on $G_{\varepsilon, \delta}$ as well as on its interior. Consequently it follows that 
		\[\frac{1}{2 \pi i } \int_{\Gamma_{\varepsilon, \delta}} \frac{f_q(z)}{z-1} \, dz = - \frac{1}{2 \pi i } \int_{C_\delta} \frac{f_q(z)}{z-1} \, dz . \]
		We note that  $C_\delta$ is an arc of the circle $|z-1|=2 \sin(\delta/2)$ of angle $\pi-\delta$. From the Fractional Residue Theorem, see e.g \cite[page 209]{Gamelin} it follows that 
		\[\lim_{\delta \to 0} \frac{1}{2 \pi i} \int_{C_\delta} \frac{f_q(z)}{z-1} \, dz = -\frac{1}{2} \Res[f_q(z)/(z-1), 1]=-\frac{1}{2}.\]
		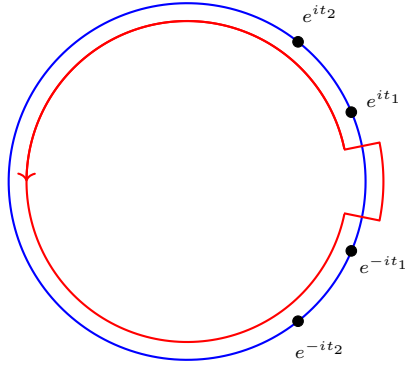
\begin{figure}
			\begin{tikzpicture}
				\begin{axis}[
					axis equal image,    
					axis lines=none,     
					ticks=none,             
					point meta min=0,
					point meta max=1,
					]
					\addplot[
					domain=0:2*pi,
					samples=200,
					thin,
					blue,
					]
					({cos(deg(x))}, {sin(deg(x))});
					
					\addplot[
					domain=0.203:2*pi-0.203,
					samples=200,
					thin,
					red,
					]
					({0.8*cos(deg(x))}, {0.8*sin(deg(x))});
					
					\addplot[
					domain=0.203:pi,
					samples=200,
					thin,
					red,
					->
					]
					({0.8*cos(deg(x))}, {0.8*sin(deg(x))});
					
					\addplot[
					domain=0.1:0.203,
					samples=200,
					thin,
					red
					]
					({1*cos(deg(x))}, {1*sin(deg(x))});

					\addplot[
					domain=-0.203:-0.1,
					samples=200,
					thin,
					red
					]
					({1*cos(deg(x))}, {1*sin(deg(x))});
					
					\addplot[
					domain=0.797:1,
					samples=200,
					thin,
					red
					]
					({x*cos(deg(0.2))}, {x*sin(deg(0.2))});
					
					\addplot[
					domain=0.797:1,
					samples=200,
					thin,
					red
					]
					({x*cos(deg(-0.2))}, {x*sin(deg(-0.2))});

					\addplot[
					domain=pi+0.01:3/2*pi-0.01,
					samples=200,
					thin,
					red,
					<-
					]
					({1+0.1*cos(deg(x))}, {0.1*sin(deg(x))});
					
					\addplot[
					domain=pi/2+0.01:3/2*pi-0.01,
					samples=200,
					thin,
					red,
					]
					({1+0.1*cos(deg(x))}, {0.1*sin(deg(x))});

					\addplot[
					domain=-0.203:-0.1,
					samples=200,
					thin,
					white
					]
					({1.1*cos(deg(x))}, {1.1*sin(deg(x))});
					
					\addplot[only marks,mark=*,color=black,mark size=2pt] coordinates {(0.92, 0.389) (0.92, -0.389) (0.6216, 0.7833) (0.6216, -0.7833) (1,0)};
					\node at (1.105, 0.467) {$\scriptstyle e^{i t_1} $};
					\node at (1.105, -0.467) {$\scriptstyle e^{-i t_1} $};
					\node at (0.7459, 0.94) {$\scriptstyle e^{i t_2} $};
					\node at (0.7459, -0.94) {$\scriptstyle e^{-i t_2} $};
					\node at (1.1, 0) {$\scriptstyle 1 $};
				\end{axis}
			\end{tikzpicture}
			\caption{The curve \CR{$G_{\varepsilon, \delta}$}  in the proof of Theorem \ref{thm:24analogi} and the \CB{unit} \CB{circle}. The $d$ zeroes of the polynomial are on the unit circle. }
			\label{Fig:thm24Analogi}
		\end{figure}
		In particular this means that
		\[	\lim_{ \delta \to 0} \frac{1}{2 \pi i } \int_{\Gamma_{\varepsilon, \delta}} \frac{f_q(z)}{z-1} \, dz =\frac{1}{2}.\]
		Next we see that
		\[\frac{1}{2 \pi i } \int_{\Gamma_{\varepsilon, \delta}} \frac{f_q(z)}{z-1} \, dz \\ = I_{1, \delta}+ I_{2, \varepsilon} + I_{3, \varepsilon},\]
		where 
		\[I_{1, \delta}=\frac{1}{2 \pi} \int_{\delta}^{t_1/2}\left( \frac{f_q(e^{i \theta})}{e^{i \theta}-1}e^{i \theta}+ \frac{f_q(e^{-i \theta})}{e^{-i \theta}-1}e^{-i \theta} \right) d\theta,\]
		
		\[I_{2, \varepsilon}=\frac{1}{2 \pi i } \int_{0}^{\varepsilon} \left(\frac{f_q((1-r)e^{-i t_1/2})}{(1-r)e^{-i t_1/2}-1} e^{-i t_1/2} - \frac{f_q((1-r)e^{i t_1/2})}{(1-r)e^{i t_1/2}-1} e^{i t_1/2} \right) dr\]
		and
		\[I_{3, \varepsilon}=\frac{1}{2\pi}  \int_{t_1/2}^{\pi}\left( \frac{f_q((1-\varepsilon)e^{i \theta})}{(1-\varepsilon)e^{i\theta}-1}(1-\varepsilon)e^{i \theta} +\frac{f_q((1-\varepsilon)e^{-i \theta})}{(1-\varepsilon)e^{-i\theta}-1}(1-\varepsilon)e^{-i \theta} \right) d\theta.\]
		We start with considering $I_{1, \delta}$. We observe that for $0<\theta<t_1/2$ we have that $f_{q}(e^{\pm i \theta}) = |P(e^{i \theta})|^q e^{ \pm i \theta q d/2}$. Consequently it follows that 
		\[\begin{split}
			\lim_{ \delta \to 0} I_{1, \delta} &= \lim_{ \delta \to 0} \frac{1}{2 \pi}
			\int_{\delta}^{t_1/2} \left(\frac{|P(e^{i \theta})|^q e^{i \theta q d/2}}{2 i \sin \theta/2}e^{i \theta /2}- \frac{|P(e^{i \theta})|^q e^{-i \theta q d/2}}{2 i \sin \theta/2}e^{-i \theta /2} \right) d\theta \\ &=\lim_{ \delta \to 0} \frac{1}{2 \pi} \int_{\delta}^{t_1/2} \frac{|P(e^{i \theta})|^q \sin{((dq/2+1)\theta)}}{\sin \theta/2} \, d\theta \\
			&=\frac{1}{2 \pi} \int_{0}^{t_1/2} \frac{|P(e^{i \theta})|^q \sin{((dq/2+1)\theta)}}{\sin \theta/2 } \, d\theta.
		\end{split}\]
		This means that for every small $\varepsilon>0$ we have
		\begin{equation}\label{eq:rep1}
			\frac{1}{2 \pi} \int_{0}^{t_1/2} \frac{|P(e^{i \theta})|^q \sin{((dq/2+1)\theta)}}{\sin \theta/2} \, d\theta+I_{2, \varepsilon}+I_{3, \varepsilon}=\frac{1}{2}.
		\end{equation}
		We now let $\varepsilon$ tend to $0$ to obtain the desired representation formula from \eqref{eq:rep1}. We can discard $I_{2, \varepsilon}$ since
		\begin{equation}\label{eq:eps0}
			\lim_{\varepsilon \to 0} \frac{1}{2 \pi i } \int_{0}^{\varepsilon} \left(\frac{f_q((1-r)e^{-i t_1/2})}{(1-r)e^{-i t_1/2}-1} e^{-i t_1/2} - \frac{f_q((1-r)e^{i t_1/2})}{(1-r)e^{i t_1/2}-1} e^{i t_1/2} \right) dr=0.
		\end{equation}
		Considering $I_{3, \varepsilon}$, we define $f_{q, \pm}(e^{ \pm i \theta})=\lim_{\varepsilon \to 0} f_q((1-\varepsilon)e^{\pm i \theta})$ for $t_1/2 \leq \theta \leq \pi$. We observe that
		\[(e^{i \theta}-e^{it_n})(e^{i \theta}-e^{-it_n})=2e^{i \theta}(\cos \theta -\cos t_n).\]
		Using this observation we see that for $t_n<\theta<t_{n+1}$ we have
		\begin{equation}\label{eq:1516}
			f_{q, \pm}(e^{\pm i \theta}) = |P(e^{i \theta})|^q e^{ \pm i \theta q d/2} e^{\mp i \pi q n}.
		\end{equation}
		Letting $\varepsilon$ tend to $0$, using \eqref{eq:1516} and continuity it follows that
		\begin{equation}\label{eq:1534}
			\begin{split}
				&\frac{1}{2\pi} \int_{t_n}^{t_{n+1}} \left( \frac{f_q((1-\varepsilon)e^{i \theta})}{(1-\varepsilon)e^{i\theta}-1}(1-\varepsilon)e^{i \theta} +\frac{f_q((1-\varepsilon)e^{-i \theta})}{(1-\varepsilon)e^{-i\theta}-1}(1-\varepsilon)e^{-i \theta} \right) d\theta \\
				\to &\frac{1}{2\pi} \int_{t_n}^{t_{n+1}} \left(\frac{|P(e^{i\theta})|^q e^{i \theta q d/2} e^{-i \pi q n}e^{i\theta}}{e^{i \theta}-1} +\frac{|P(e^{i\theta})|^q e^{-i \theta q d/2} e^{i \pi q n}e^{-i\theta}}{e^{-i \theta}-1} \right) d\theta\\
				=&\frac{1}{2\pi} \int_{t_n}^{t_{n+1}} |P(e^{i \theta})|^q\frac{\sin \left(\frac{dq+1}{2}\theta-\pi q n\right)}{\sin \frac{\theta}{2}}\, d\theta
		\end{split} \end{equation}
		for $n \geq 1$, and for $n=0$ with $t_1/2$ and $t_1$ as the lower and upper bounds of integration. 
		Consequently, combining \eqref{eq:rep1}, \eqref{eq:eps0} and \eqref{eq:1534} concludes the proof for $d$ even. 
		
		Now assume $d$ is odd. Then \[P(z)=C(z+1)\prod_{j=1}^{k}(z-e^{it_j})(z-e^{-it_j}).\]
		We observe that $1+e^{ix}=\sqrt{2(1+\cos x)}e^{ix/2}$. Thus in this case we define
		\[f_q(z)=C^q(z+1)^q\prod_{j=1}^{k}((z-e^{it_j})(z-e^{-it_j}))^q,\]
		and $f_{q, \pm}(e^{ \pm i \theta})=\lim_{\varepsilon \to 0} f_q((1-\varepsilon)e^{\pm i \theta})$ for $t_1/2 \leq \theta \leq \pi$. Thus also for $d$ odd and $t_n<\theta< t_{n+1}$ we have
		\[f_{q, \pm}(e^{\pm i \theta}) = |P(e^{i \theta})|^q e^{ \pm i \theta q d/2} e^{\mp i \pi q n},\] and we can repeat the proof given above. 
	\end{proof}
	Our goal is to use Theorem \ref{thm:24analogi} to derive the upper bound on $\mathscr{C}_{d,p}$ stated in Theorem \ref{thm:main}. It is clear from the representation formula in Theorem \ref{thm:24analogi} that this requires a certain knowledge on the zeroes of the extremal function. Section \ref{sec:tn} is devoted to establishing this information. 
	
	\section{The zeroes of the extremal function}\label{sec:tn}
	In this section we extract information about the arguments of the zeroes $t_n$ of the extremal function $\varphi_{d, p}$.
	We begin by stating a result of Turán \cite[Theorem 3]{Turan} which yields a lower bound for the argument $t_1$ of the first zero.
	\begin{lemma}[Turán]\label{lem:Turan}
		Let $1 <p < \infty$, and let $t_1$ denote the smallest argument of a zero of the polynomial $\varphi_{d, p}$ of degree $d \geq 1$. Then
		\[t_1 \geq \frac{\pi}{d}.\]
	\end{lemma}
	We also have the following simple lemma which is key in proving $\mathscr{C}_{4,p} \leq 2p+1$.
	
	\begin{lemma}\label{lem:d4pi2}
		Let $d=4$ and assume $p > 1$. Let $\varphi_{d,p}$ be the extremal function and denote its zeroes by $e^{it_1}, e^{i t_2}, e^{-it_1}$ and $e^{-it_2}$, where $0 <t_1 \leq t_2 \leq \pi$. Then $t_2 \geq \pi /2$. 
	\end{lemma}
	
	\begin{proof}
	Let $P(z)$ be a polynomial of degree $4$ such that $P(1)=1$ and let the zeroes of $P$ be given by $e^{i \tau_1}$, $e^{-i \tau_1}$, $e^{i \tau_2}$ and $e^{-i \tau_2}$, where $0<\tau_1 \leq \tau_2 <\pi/2$. Then 
		\[\begin{split}
			P(e^{i \theta})&= \frac{(e^{i \theta}-e^{i \tau_1})(e^{i \theta}-e^{-i \tau_1})(e^{i \theta}-e^{i \tau_2})(e^{i \theta}-e^{-i \tau_2})}{(1-e^{i \tau_1})(1-e^{-i \tau_1})(1-e^{i \tau_2})(1-e^{-i \tau_2})}\\&=e^{2 i \theta} \frac{(\cos \theta-\cos \tau_1)(\cos \theta - \cos \tau_2)}{(1-\cos \tau_1)(1-\cos \tau_2)}.
		\end{split}
		\]
		Next we let $\gamma_1=\pi-\tau_1$ and $\gamma_2=\pi- \tau_2$ and $Q$ be the polynomial given by
		\[\begin{split}
			Q(e^{i \theta})&= \frac{(e^{i \theta}-e^{i \gamma_1})(e^{i \theta}-e^{-i \gamma_1})(e^{i \theta}-e^{i \gamma_2})(e^{i \theta}-e^{-i \gamma_2})}{(1-e^{i \gamma_1})(1-e^{-i \gamma_1})(1-e^{i \gamma_2})(1-e^{-i \gamma_2})}\\&=e^{2 i \theta} \frac{(\cos \theta-\cos \gamma_1)(\cos \theta - \cos \gamma_2)}{(1-\cos \gamma_1)(1-\cos \gamma_2)}.
		\end{split}
		\]
		We observe that $Q$ is of degree $4$ and $Q(1)=1$. From \eqref{eq:extprob} we see that the proof is complete  if we can show that $\|Q\|_p^p \leq \|P\|_p^p$, since this shows that $P$ cannot be the extremal function.
		Calculating the norm and using the substitution $u=\pi-\theta$ we see that
		\[\begin{split}
			\|Q\|_p^p&= \int_{-\pi}^{\pi} |Q(e^{i \theta})|^p\, \frac {d\theta}{2 \pi}=\frac{1}{\pi} \int_{0}^{\pi}\left|\frac{(\cos \theta-\cos \gamma_1)(\cos \theta - \cos \gamma_2)}{(1-\cos \gamma_1)(1-\cos \gamma_2)}\right|^p \, d\theta\\&=
			\frac{1}{\pi} \int_{0}^{\pi}\left|\frac{(\cos \theta+\cos \tau_1)(\cos \theta + \cos \tau_2)}{(1+\cos \tau_1)(1+\cos \tau_2)}\right|^p \, d\theta\\&=\frac{1}{\pi} \int_{0}^{\pi}\left|\frac{(\cos \theta-\cos \tau_1)(\cos \theta - \cos \tau_2)}{(1+\cos \tau_1)(1+\cos \tau_2)}\right|^p \, d\theta \\ &\leq 
			\frac{1}{\pi} \int_{0}^{\pi}\left|\frac{(\cos \theta-\cos \tau_1)(\cos \theta - \cos \tau_2)}{(1-\cos \tau_1)(1-\cos \tau_2)}\right|^p \, d\theta \\
			&=\|P\|_p^p.
		\end{split}\]
		Note that the inequality follows since $0 < \tau_1 \leq \tau_2 <\pi/2$ guarantees that $\cos \tau_1$ and $\cos \tau_2$ both take values in $(0,1)$.
	\end{proof}
	
	It is easy to see that we can generalize Lemma \ref{lem:d4pi2} to say that if $d$ is even, then at least one zero has an argument greater than or equal to $\pi/2$. 
	\section{Proof of Theorem \ref{thm:main}}\label{sec:mainproof}
	Before embarking on the proof of Theorem \ref{thm:main}, we will introduce some necessary terminology. Given an integer $d \geq 2$ we let $k=\lfloor d/2 \rfloor$ and $T_d$ be the set of non-decreasing sequences with $k+2$ elements such that $\tau_0=0$ and $\tau_{k+1}=\pi$. Moreover, we let $t=(t_n)_{n=0}^{k+1}$ be the sequence of arguments in $(0, \pi]$ of the zeroes of the extremal function $\varphi_{d, p}$, extended with the element $t_0=0$, as well as $t_{k+1}= \pi$ if $d$ is even. It is clear that $t$ belongs to $T_d$.
	
	Given an integer $d \geq 2$ and a sequence $\tau$ in $T_d$ we define the function
	\[K_{d,p}(\tau; \theta)=\sum_{n=0}^{k} \chi_{(\tau_n,\tau_{n+1})}(\theta) \frac{\sin ((\frac{dp}{4}+\frac{1}{2})\theta-\frac{p}{2} \pi n)}{\sin\frac{\theta}{2}},\]
	and let $M_{d,p}(\tau; \theta)=\max(K_{d,p}(\tau; \theta),0)$. Further we define 
	\[E_{d,p}(\tau)=\frac{1}{\pi} \int_{0}^{\pi} M_{d,p}^2(\tau; \theta) \, d \theta.\]
	
	We now establish the following result, which will be key in proving Theorem \ref{thm:main}. Note that this can be viewed as the trigonometric analogue of \cite[Corollary 6.1]{Brevig}. 
	\begin{theorem}\label{thm:cor61analogi}
		Let $2 \leq p < \infty$, let $d \geq 2$ be an integer and define $k=\lfloor d/2 \rfloor$. Let $\Lambda$ be a subset of $T_d$ containing the sequence $t=(t_n)_{n=0}^k$. 
		Then
		\[\mathscr{C}_{d,p} \leq \sup_{\tau \in \Lambda} E_{d,p}(\tau).\]
	\end{theorem}
	
	\begin{proof}
		Applying Theorem \ref{thm:24analogi} with $P(z)=\varphi_{d,p}(z)$ and $q=p/2$, we get
		\[1=|\varphi_{d,p}(1)|^{p/2}=\frac{1}{\pi} \int_{0}^{\pi}|\varphi_{d,p}(e^{i \theta})|^{p/2} K_{d,p}(t; \theta) \, d \theta.\]
		Clearly it also holds that 
		\begin{equation}\label{eq:61eq}
			1 \leq \frac{1}{\pi} \int_{0}^{\pi}|\varphi_{d,p}(e^{i \theta})|^{p/2} M_{d,p}(t; \theta) \, d \theta.
		\end{equation}
		Squaring both sides of \eqref{eq:61eq} and using the Cauchy--Schwarz inequality we obtain
		\[1 \leq \|\varphi_{d,p}\|_p^p \frac{1}{\pi} \int_{0}^{\pi} M_{d,p}^2(t; \theta) \, d \theta. \]
		Thus by the definition of $\mathscr{C}_{d,p}$ we conclude that 
		\[\mathscr{C}_{d,p} \leq \frac{1}{\pi} \int_{0}^{\pi} M_{d,p}^2(t; \theta) \, d \theta \leq \sup_{\tau \in \Lambda} \frac{1}{\pi} \int_{0}^{\pi} M_{d,p}^2(\tau; \theta) \, d \theta. \qedhere\]
	\end{proof}
	Note that in the proof of Theorem \ref{thm:cor61analogi} we replaced $K_{d,p}(\tau; \theta)$ by its positive part, and we also applied the Cauchy--Schwarz inequality. It follows that the upper bound for $\mathscr{C}_{d,p}$ obtained with this technique cannot be optimal for $p \neq 2$.
	
	Using Theorem \ref{thm:cor61analogi}, we will prove Theorem \ref{thm:main} by determining the supremum in
	\begin{equation}\label{eq:1546}
		\mathscr{C}_{d,p} \leq \sup_{\tau \in \Lambda} E_{d,p}(\tau),
	\end{equation} where $\Lambda \subseteq T_d$ is a set of sequences of possible arguments of the zeroes of the extremal function $\varphi_{d,p}$, for $2 \leq d \leq 4$ and $ 2 \leq p \leq 4$. The argument follows the lines of what was done in \cite{Instanes}. In particular, due to Lemma \ref{lem:Turan}, for any integer $d$ we can restrict $\Lambda$ to the set $\Lambda_d$, where we define
	\begin{equation}\label{eq:Lambdad}
		\Lambda_d=\{(\tau_n)_{n=0}^{k+1}: 0=\tau_0 \leq \pi/d \leq \tau_1 \leq \tau_2 \leq \dots \leq \tau_{k+1} = \pi \}. 
	\end{equation}
	In the special case $d=4$ we use Lemma \ref{lem:d4pi2} to restrict $\Lambda$ further and redefine $\Lambda_4$ as
	\begin{equation}\label{eq:Lambda4}
		\Lambda_4=\{\tau=(\tau_n)_{n=0}^3: \tau_0=0, \, \frac{\pi}{d} \leq \tau_1 \text{ and } \frac{\pi}{2} \leq \tau_2 \leq \tau_3 =\pi\}.
	\end{equation}
	
	To determine the supremum in \eqref{eq:1546} we introduce terminology similar to that in \cite{Brevig} and \cite{Instanes}. We refer to each connected component of the set
	\[\left\{\theta: \sin \left(\left(\frac{dp}{4}+\frac{1}{2}\right) \theta - \frac{p}{2} \pi n \right)>0 \text{ and } 0 \leq \theta \leq \pi \right\}\]
	as an \emph{interval at level $n$}. We note that each interval at level $n$ has length $l= 4\pi/(dp+2)$, and that the left endpoint of an interval at level $n$ is always of the form $l(pn+4j)/2$ for some integer $j$. Furthermore, it is easy to see that when $2<p<4$, an interval at level $n$ intersects exactly one interval at level $n+1$. Thus one can think of determining the supremum in \eqref{eq:1546} as finding the correct "jumps" between intervals at level $n$ and level $n+1$ in order to maximize the corresponding integral $E_{d,p}(\tau)$. See Figure \ref{fig:d3p103} for an illustration. Below we state some lemmas that allow us to exclude certain sequences from $\Lambda_d$. In particular we start with a lemma saying that we need not consider sequences where $\tau_n \geq lpn/2$. We recall that $lpn/2$ is a left endpoint of an interval at level $n$, and that $l=4\pi/(dp+2)$.
	\begin{figure}
		\centering
		\begin{tikzpicture}[scale=1.5]
			\begin{axis}[
				axis equal image,
				axis lines = none,
				trig format plots=rad]
				
				
				\addplot[thin, name path=t0] coordinates {(0,0) (pi/3,0)};
				\addplot[domain=0:pi/3, samples=100, color=black!75, thin, name path=b0] ({x},{1/5*sin(3*x))*sin(3*x)});
				\addplot[red!50] fill between [of=b0 and t0];	
				
				\addplot[thin] coordinates {(2*pi/3,0) (pi,0)};
				\addplot[domain=2*pi/3:pi, samples=100, color=black!75, thin] ({x},{1/5*sin(3*x))*sin(3*x)});
				
				\addplot[thin] coordinates {(0,-0.4) (2/9*pi,-0.4)};
				\addplot[domain=0:2/9*pi, samples=100, color=black!75, thin] ({x},{1/5*sin(3*x-5/3*pi))*sin(3*x-5/3*pi)-0.4});
				
				\addplot[thin, name path=t1] coordinates {(5/9*pi,-0.4) (8/9*pi,-0.4)};
				\addplot[domain=5/9*pi:8/9*pi, samples=100, color=black!75, thin, name path=b1] ({x},{1/5*sin(3*x-5/3*pi))*sin(3*x-5/3*pi)-0.4});
				\addplot[red!50] fill between [of=b1 and t1];

				\addplot[thin,color=blue,->] coordinates {(pi/3,0) (pi/3,-0.4)};
				
				\addplot[thin, color=black!0] coordinates {(0,-0.55) (3.5,-0.55)};
				\addplot[only marks,mark=|,color=black,mark size=2pt] coordinates {(2*pi/9,-0.4) (pi/3,-0.4) (5*pi/9,-0.4) (pi,-0.4)};
				\node at (axis cs: 2*pi/9,-0.5) {$\scriptstyle \frac{2}{9}\pi$};
				\node at (axis cs: pi/3,-0.5) {$\scriptstyle \frac{1}{3}\pi$};
				\node at (axis cs: 5*pi/9,-0.5) {$\scriptstyle  \frac{5}{9}\pi$};
				\node at (axis cs: pi,-0.5) {$\scriptstyle \pi$};
			\end{axis}
		\end{tikzpicture}
		\caption{An illustration of the intervals at level $0$ and level $1$ for $d=3$ and $p=10/3$. Theorem \ref{thm:uppergood} tells us that the supremum in equation \eqref{eq:1546} is attained whenever $\pi/3 \leq \tau_1 \leq 5\pi/9$. The shaded area represents $E_p(\tau)$ without considering the integrand factor $1/\sin^2(\theta/2)$. }
		\label{fig:d3p103}

		\centering
		\begin{tikzpicture}[scale=1.5]
			\begin{axis}[
				axis equal image,
				axis lines = none,
				trig format plots=rad]
				
				
				\addplot[thin] coordinates {(0,0) (pi/3,0)};
				\addplot[domain=0:pi/3, samples=100, color=black!75, thin] ({x},{1/5*sin(3*x))*sin(3*x)});
				
				\addplot[thin, name path=t0] coordinates {(0,0) (3*pi/18,0)};
				\addplot[domain=0:3*pi/18, samples=100, color=black!75, thin, name path=b0] ({x},{1/5*sin(3*x))*sin(3*x)});
				\addplot[red!50] fill between [of=b0 and t0];	
				
				\addplot[thin] coordinates {(2*pi/3,0) (pi,0)};
				\addplot[domain=2*pi/3:pi, samples=100, color=black!75, thin] ({x},{1/5*sin(3*x))*sin(3*x)});
				
				\addplot[thin] coordinates {(0,-0.4) (2/9*pi,-0.4)};
				\addplot[domain=0:2/9*pi, samples=100, color=black!75, thin] ({x},{1/5*sin(3*x-5/3*pi))*sin(3*x-5/3*pi)-0.4});
				
				\addplot[thin, name path=t11] coordinates {(3/18*pi,-0.4) 	(2/9*pi,-0.4)};
				\addplot[domain=3/18*pi:2/9*pi, samples=100, color=black!75, thin, name path=b11] ({x},{1/5*sin(3*x-5/3*pi))*sin(3*x-5/3*pi)-0.4});
				\addplot[red!50] fill between [of=b11 and t11];	
				
				\addplot[thin, name path=t1] coordinates {(5/9*pi,-0.4) 	(8/9*pi,-0.4)};
				\addplot[domain=5/9*pi:8/9*pi, samples=100, color=black!75, thin, name path=b1] ({x},{1/5*sin(3*x-5/3*pi))*sin(3*x-5/3*pi)-0.4});
				\addplot[red!50] fill between [of=b1 and t1];	
				
				\addplot[thin,color=blue,->] coordinates {(3*pi/18,0) (3*pi/18,-0.4)};
				
				\addplot[thin, color=black!0] coordinates {(0,-0.55) (3.5,-0.55)};
				\addplot[only marks,mark=|,color=black,mark size=2pt] coordinates {(pi/9,-0.4) (2*pi/9,-0.4) (pi/3,-0.4) (5*pi/9,-0.4) (pi,-0.4)};
				\node at (axis cs: pi/9,-0.5) {$\scriptstyle \frac{1}{9}\pi$};
				\node at (axis cs: 2*pi/9,-0.5) {$\scriptstyle \frac{2}{9}\pi$};
				\node at (axis cs: pi/3,-0.5) {$\scriptstyle \frac{1}{3}\pi$};
				\node at (axis cs: 5*pi/9,-0.5) {$\scriptstyle  \frac{5}{9}\pi$};
				\node at (axis cs: pi,-0.5) {$\scriptstyle \pi$};
			\end{axis}
		\end{tikzpicture}
		\caption{The shaded area represents $E_p(\gamma)$ for $\gamma_1=3\pi/18$ without considering the integrand factor $1/\sin^2(\theta/2)$ for $d=3$ and $p=10/3$. Keep in mind that $\pi/9$ is the midpoint between the first interval at level $0$ and the first interval at level $1$. Lemma \ref{lem:mid} tells us that for any $\gamma_1$ in $[\pi/9, \pi/3]$ it will follow that $E_p(\gamma) \leq E_p(\tau)$ where $\tau_1= \pi/3$ as seen in Figure \ref{fig:d3p103}. }
		\label{fig:d3p103_part2}
	\end{figure}
	
	\begin{lemma}\label{lem:tleq}
		Let $2 \leq p \leq 4$ and $d\geq 2$ be an integer. Let $\gamma$ be a sequence in $\Lambda_d$ and let $\tau$ be the sequence $\tau_n=\min(lpn/2, \gamma_n)$. Then $\tau$ is in $\Lambda_d$ as well and $E_p(\gamma) \leq E_p(\tau)$.  
	\end{lemma}
	
	\begin{proof}
		We see that $\tau$ is an increasing sequence and since $lp/2 \geq \pi/d$ it follows that $\tau_1 \geq \pi/d$. Thus if $d \neq 4$, then $\tau$ is in $\Lambda_d$. Further for $d=4$ we see that $pl \geq \pi/2$ so $\tau$ is in $\Lambda_d$ also if $d=4$. If $\gamma_n=\tau_n$ for all $n$ we are done. Thus we fix some integer $m$ such that $\gamma_m > \tau_m$. Let $j$ be the largest integer such that $l(pm+4j)/2 \leq \gamma_m$. (That is $l(pm+4j)/2$ is the largest left endpoint of an interval at level $m$ smaller than $\gamma_m$.) By a pointwise estimate and using that $\sin(\theta/2)$ is an increasing function for $0<\theta < \pi$ it follows that replacing $\gamma_m$ with $l(pm+4j)/2$ will increase the value of $E_p(\gamma)$. Further by periodicity and using that $\sin(\theta/2)$ is an increasing function it follows that replacing $\gamma_m$ with $mpl/2$ will increase the value of $E_{d,p}(\gamma)$ further. Since this holds for all $m$ such that $\gamma_m > \tau_m$ we conclude that $E_p(\gamma) \leq E_p(\tau)$. 
	\end{proof}
	
	In particular we note that the above lemma allows us to impose the condition $\tau_n \leq npl/2$ for all $n$ on the set $\Lambda$ in the supremum \eqref{eq:1546}. Next we state a lemma that essentially tells us that the supremum \eqref{eq:1546} is not attained if we make a "jump" from level $n$ to level $n+1$ between the right endpoint of an interval at level $n$ and the midpoint between the interval at level $n$ and its intersecting interval at level $n+1$. This is illustrated in Figure \ref{fig:d3p103} and Figure \ref{fig:d3p103_part2}. 
	
	\begin{lemma}\label{lem:mid}
		Let $2 \leq p \leq 4$ and $d\geq 2$ be an integer. Let $\gamma$ be a sequence in $\Lambda_d$ such that $\gamma_n \leq nlp/2$. Fix an integer $r$. Assume $M \leq \gamma_{r+1} \leq b$, where $(a,b)$ is an interval at level $r$ and $M$ is the midpoint between $(a,b)$ and the intersecting interval $(c,d)$ at level $r+1$. Let $\tau$ be the sequence in $T_d$ defined by $\tau_{r+1}=b$ and $\tau_n=\gamma_n$ for all $n \neq r+1$. Then $E_p(\gamma) \leq E_p(\tau)$. If $d \leq 4$ then it also holds that $\tau$ is in $\Lambda_d$. 
	\end{lemma}
	
	\begin{proof}
		We need only show
		\[\int_{M}^{b} \left(\frac{\sin \frac{dp+2}{4}\theta-\frac{p}{2}\pi r}{\sin \frac{\theta}{2}}\right)^2 \, d\theta \geq \int_{M}^{d} \left(\frac{\sin \frac{dp+2}{4}\theta-\frac{p}{2}\pi (r+1)}{\sin \frac{\theta}{2}} \right)^2 d\theta,\]
		which follows from the pointwise estimate 
		\[\left(\frac{\sin \frac{dp+2}{4}\theta-\frac{p}{2}\pi r}{\sin \frac{\theta}{2}}\right)^2 \geq \left(\frac{\sin \frac{dp+2}{4}\theta-\frac{p}{2}\pi (r+1)}{\sin \frac{\theta}{2}} \right)^2\] for all $M \leq \theta \leq d$. It remains to show that $\tau$ is in $\Lambda_d$ whenever $d \leq 4$. If $d=2$ or $d=3$ then $0=\tau_0 \leq \tau_1 \leq \tau_2=\pi$ and clearly $\tau_1 \geq \pi/d$ since $\gamma_1 \geq \pi/d$, so $\tau$ is in $\Lambda_d$. Next consider $d=4$. Then $\pi/d \leq \tau_1 \leq pl/2 \leq \pi/2 \leq \tau_2$, and so $\tau$ is in $\Lambda_d$. 
	\end{proof}
	
	\begin{figure}
		\centering
		\begin{tikzpicture}[scale=1.5]
			\begin{axis}[
				axis equal image,
				axis lines = none,
				trig format plots=rad]
				
				|\node at (axis cs: -0.2,0) {$\scriptstyle 0$};
				\node at (axis cs: -0.2,-0.4) {$\scriptstyle 1$};
				\node at (axis cs: -0.2,-0.8) {$\scriptstyle 2$};
				
				\addplot[thin, name path=t0] coordinates {(0,0) (pi/4,0)};
				\addplot[domain=0:pi/4, samples=100, color=black!75, thin, name path=b0] ({x},{1/5*sin(4*x))*sin(4*x)});
				\addplot[red!50] fill between [of=b0 and t0];	
				
				\addplot[thin] coordinates {(pi/2,0) (3/4*pi,0)};
				\addplot[domain=pi/2:3*pi/4, samples=100, color=black!75, thin] ({x},{1/5*sin(4*x))*sin(4*x)});
				
				\addplot[thin] coordinates {(0,-0.4) (3/16*pi,-0.4)};
				\addplot[domain=0:3/16*pi, samples=100, color=black!75, thin] ({x},{1/5*sin(4*x-7/4*pi))*sin(4*x-7/4*pi)-0.4});
				
				\addplot[thin, name path=t1] coordinates {(7/16*pi,-0.4) (11/16*pi,-0.4)};
				\addplot[domain=7/16*pi:11/16*pi, samples=100, color=black!75, thin, name path=b1] ({x},{1/5*sin(4*x-7/4*pi))*sin(4*x-7/4*pi)-0.4});
				\addplot[red!50] fill between [of=b1 and t1];	
				
				\addplot[thin] coordinates {(15/16*pi,-0.4) (pi,-0.4)};
				\addplot[domain=15/16*pi:pi, samples=100, color=black!75, thin] ({x},{1/5*sin(4*x-7/4*pi))*sin(4*x-7/4*pi)-0.4});
				
				\addplot[thin] coordinates {(0,-0.8) (2/16*pi,-0.8)};
				\addplot[domain=0:2/16*pi, samples=100, color=black!75, thin] ({x},{1/5*sin(4*x-7/2*pi))*sin(4*x-7/2*pi)-0.8});
				
				\addplot[thin] coordinates {(6/16*pi,-0.8) (10/16*pi,-0.8)};
				\addplot[domain=6/16*pi:10/16*pi, samples=100, color=black!75, thin] ({x},{1/5*sin(4*x-7/2*pi))*sin(4*x-7/2*pi)-0.8});
				
				\addplot[thin, name path=t2] coordinates {(14/16*pi,-0.8) (pi,-0.8)};
				\addplot[domain=14/16*pi:pi, samples=100, color=black!75, thin, name path=b2] ({x},{1/5*sin(4*x-7/2*pi))*sin(4*x-7/2*pi)-0.8});
				\addplot[red!50] fill between [of=b2 and t2];	
				
				\addplot[thin,color=blue,->] coordinates {(pi/4,0) (pi/4,-0.4)};
				\addplot[thin,color=blue,->] coordinates {(11*pi/16,-0.4) (11*pi/16,-0.8)};
				
				\addplot[thin, color=black!0] coordinates {(0,-1) (3.5,-1.1)};
				\addplot[only marks,mark=|,color=black,mark size=2pt] coordinates {(pi/4,-0.8) (pi/2,-0.8) (3/4*pi,-0.8) (pi,-0.8)};
				\node at (axis cs: pi/4,-1) {$\scriptstyle \pi/4$};
				\node at (axis cs: pi/2,-1) {$\scriptstyle \pi/2$};
				\node at (axis cs: 3*pi/4,-1) {$\scriptstyle 3\pi/4$};
				\node at (axis cs: pi,-1) {$\scriptstyle \pi$};
			\end{axis}
		\end{tikzpicture}
		\caption{Together with Figure \ref{fig:d4_cand2} this gives an illustration of Lemma \ref{lem:d4p34} for $d=4$ and $p=3.5$. Here we see the sequence $\tau$ from Lemma \ref{lem:d4p34}. The shaded area represents $E_p(\tau)$ without considering the integrand factor $1/\sin^2(\theta/2)$. The lemma tells us that $E_p(\gamma) \leq E_p(\tau)$, where $\gamma$ is illustrated in Figure \ref{fig:d4_cand2}.}
		\label{fig:d4_cand1}
		\centering
		\begin{tikzpicture}[scale=1.5]
			\begin{axis}[
				axis equal image,
				axis lines = none,
				trig format plots=rad]
				
				\node at (axis cs: -0.2,0) {$\scriptstyle 0$};
				\node at (axis cs: -0.2,-0.4) {$\scriptstyle 1$};
				\node at (axis cs: -0.2,-0.8) {$\scriptstyle 2$};
				
				\addplot[thin, name path=t0] coordinates {(0,0) (pi/4,0)};
				\addplot[domain=0:pi/4, samples=100, color=black!75, thin, name path=b0] ({x},{1/5*sin(4*x))*sin(4*x)});
				\addplot[red!50] fill between [of=b0 and t0];	
				
				\addplot[thin] coordinates {(pi/2,0) (3/4*pi,0)};
				\addplot[domain=pi/2:3*pi/4, samples=100, color=black!75, thin] ({x},{1/5*sin(4*x))*sin(4*x)});
				
				\addplot[thin] coordinates {(0,-0.4) (3/16*pi,-0.4)};
				\addplot[domain=0:3/16*pi, samples=100, color=black!75, thin] ({x},{1/5*sin(4*x-7/4*pi))*sin(4*x-7/4*pi)-0.4});
				
				\addplot[thin] coordinates {(7/16*pi,-0.4) (11/16*pi,-0.4)};
				\addplot[domain=7/16*pi:11/16*pi, samples=100, color=black!75, thin] ({x},{1/5*sin(4*x-7/4*pi))*sin(4*x-7/4*pi)-0.4});
				
				\addplot[thin, name path=t1] coordinates {(7/16*pi,-0.4) (8/16*pi,-0.4)};
				\addplot[domain=7/16*pi:8/16*pi, samples=100, color=black!75, thin, name path=b1] ({x},{1/5*sin(4*x-7/4*pi))*sin(4*x-7/4*pi)-0.4});
				\addplot[red!50] fill between [of=b1 and t1];	
				
				\addplot[thin] coordinates {(15/16*pi,-0.4) (pi,-0.4)};
				\addplot[domain=15/16*pi:pi, samples=100, color=black!75, thin] ({x},{1/5*sin(4*x-7/4*pi))*sin(4*x-7/4*pi)-0.4});
				
				\addplot[thin] coordinates {(0,-0.8) (2/16*pi,-0.8)};
				\addplot[domain=0:2/16*pi, samples=100, color=black!75, thin] ({x},{1/5*sin(4*x-7/2*pi))*sin(4*x-7/2*pi)-0.8});
				
				\addplot[thin] coordinates {(6/16*pi,-0.8) (10/16*pi,-0.8)};
				\addplot[domain=6/16*pi:10/16*pi, samples=100, color=black!75, thin] ({x},{1/5*sin(4*x-7/2*pi))*sin(4*x-7/2*pi)-0.8});
				
				\addplot[thin, name path=t22] coordinates {(8/16*pi,-0.8) (10/16*pi,-0.8)};
				\addplot[domain=8/16*pi:10/16*pi, samples=100, color=black!75, thin, name path=b22] ({x},{1/5*sin(4*x-7/2*pi))*sin(4*x-7/2*pi)-0.8});
				\addplot[red!50] fill between [of=b22 and t22];	
				
				\addplot[thin, name path=t2] coordinates {(14/16*pi,-0.8) (pi,-0.8)};
				\addplot[domain=14/16*pi:pi, samples=100, color=black!75, thin, name path=b2] ({x},{1/5*sin(4*x-7/2*pi))*sin(4*x-7/2*pi)-0.8});
				\addplot[red!50] fill between [of=b2 and t2];	
				
				\addplot[thin,color=blue,->] coordinates {(pi/4,0) (pi/4,-0.4)};
				\addplot[thin,color=blue,->] coordinates {(pi/2,-0.4) (pi/2,-0.8)};
				
				\addplot[thin, color=black!0] coordinates {(0,-1) (3.5,-1.1)};
				\addplot[only marks,mark=|,color=black,mark size=2pt] coordinates {(pi/4,-0.8) (pi/2,-0.8) (3/4*pi,-0.8) (pi,-0.8)};
				\node at (axis cs: pi/4,-1) {$\scriptstyle \pi/4$};
				\node at (axis cs: pi/2,-1) {$\scriptstyle \pi/2$};
				\node at (axis cs: 3*pi/4,-1) {$\scriptstyle 3\pi/4$};
				\node at (axis cs: pi,-1) {$\scriptstyle \pi$};
			\end{axis}
		\end{tikzpicture}
		\caption{Together with Figure \ref{fig:d4_cand1} this gives an illustration of Lemma \ref{lem:d4p34} for $d=4$ and $p=3.5$. Here we see the sequence $\gamma$ with $\gamma_2=\pi/2$. The shaded area represents $E_p(\gamma)$ without considering the integrand factor $1/\sin^2(\theta/2)$. Lemma \ref{lem:d4p34} tells us that $E_p(\gamma) \leq E_p(\tau)$, where $\tau$ is illustrated in Figure \ref{fig:d4_cand1}.} 
		\label{fig:d4_cand2}
	\end{figure}
	
	To prove Theorem \ref{thm:main} for $d=4$ and $3 < p \leq 4$ we need the following lemma, which says that if the supremum in \eqref{eq:1546} is attained by a sequence with $\tau_2 \geq \pi/2$, then necessarily $\tau_2$ is greater than the right endpoint of an interval at level $n$ namely $l(p/2+1)$. See Figure 6 and Figure 7 for an illustration.  
	
	\begin{lemma}\label{lem:d4p34}
		Let $3<p \leq 4$. Let $\gamma$ be a sequence in $\Lambda_4$ such that $\pi/2 \leq \gamma_2 \leq l(p/2+1)$. Let $\tau$ be the sequence in $T_4$ with $\tau_1=\gamma_1$ and $\tau_2=l(p/2+1)$. Then $\tau$ is in $\Lambda_4$ and $E_4(\gamma) \leq E_4(\tau)$. 
	\end{lemma}
	\begin{proof}
		It is clear that $\tau$ is in $\Lambda_4$ since $\tau_2 \geq \gamma_2$, so we only need to show $E_4(\gamma) \leq E_4(\tau)$. By a pointwise estimate it suffices to consider $\gamma_2= \pi/2$. Since all other integral parts of $E_p(\tau)$ and $E_p(\gamma)$ will be equal, it suffices to show that
		\[\int_{\pi/2}^{l(p/2+1)} \left(\frac{\sin(\frac{2p+1}{2}\theta-\frac{p}{2}\pi)}{\sin \frac{\theta}{2}}\right)^2 \, d\theta \geq \int_{\pi/2}^{l(p-1)} \left(\frac{\sin(\frac{2p+1}{2}\theta-p\pi)}{\sin \frac{\theta}{2}}\right)^2 \, d\theta.\]
		Since the integrands are positive and $l(p-1) \leq l(p/2+1)$ it suffices to show $f(p) \geq 0$, where
		\[f(p)=\int_{\pi/2}^{l(p/2+1)} \left(\frac{\sin(\frac{2p+1}{2}\theta-\frac{p}{2}\pi)}{\sin \frac{\theta}{2}}\right)^2 - \left(\frac{\sin(\frac{2p+1}{2}\theta-p\pi)}{\sin \frac{\theta}{2}}\right)^2\, d\theta.\]
		We recall that since $d=4$ then $l=2\pi/(2p+1)$. Therefore by making the substitution $u=(\theta-\pi/2)(2p+1)$ and then using trigonometric identities we see that
		\[\begin{split}
			f(p)&=\frac{1}{2p+1} \int_{0}^{3 \pi/2} \frac{\big(\sin(\frac{\theta}{2}+\frac{\pi}{4})\big)^2-\big(\sin(\frac{\theta}{2}+\frac{\pi}{4}-\frac{\pi}{2}p)\big)^2}{\big(\sin( \frac{\theta}{4p+2}+\frac{\pi}{4})\big)^2} \, d\theta\\
			&=\frac{2\sin(-p\pi/2)}{2p+1}\int_{0}^{3\pi/2} \frac{\sin \theta \sin (-p\pi/2)-\cos \theta \cos (p\pi/2)}{1+\sin \frac{\theta}{2p+1}} \, d\theta.
		\end{split}\]
		It follows that $f(p)\geq 0$ since
		\[\int_{0}^{3 \pi/2} \frac{\sin \theta}{1+\sin \frac{\theta}{2p+1}} \, d\theta \geq \int_{0}^{\pi} \frac{\sin \theta}{1+\sin \frac{\theta}{2 \cdot 3+1}} \, d\theta +\int_{\pi}^{3 \pi/2} \frac{\sin \theta}{1+\sin \frac{\theta}{2 \cdot 4+1}} \, d\theta \geq 0,\]
		and
		\[\int_{0}^{3 \pi/2} \frac{-\cos \theta}{1+\sin \frac{\theta}{2p+1}} \, d\theta \geq \int_{0}^{\pi/2} \frac{-\cos \theta}{1+\sin \frac{\theta}{2 \cdot 4+1}} \, d\theta +\int_{\pi/2}^{3 \pi/2} \frac{\sin \theta}{1+\sin \frac{\theta}{2 \cdot 3+1}} \, d\theta \geq 0. \qedhere \]
	\end{proof}
	To prove the upper bound from Theorem \ref{thm:main} we in fact prove a slightly stronger bound given in the theorem below. Figure \ref{fig:upd234} shows the plot of these upper bounds, and Figures \ref{fig:d3p103} and \ref{fig:d4_cand1} illustrate the bounds with the corresponding sequences in $\Lambda_d$. We recall that $l=4\pi/(dp+2)$ is the length of an interval at level $n$. (Below we use the notation $\tau^d$ to indicate the dependence on $d$ in the sequence, not to be confused with an exponent.)
	\begin{theorem}\label{thm:uppergood}
		Let $2 \leq p \leq 4$ and $2 \leq d \leq 4$. Then $\mathscr{C}_{d,p} \leq E_p(\tau^d)$, where 
		\[\begin{split}
			\tau^2&=(0,\tau_1^2, \pi),\\
			\tau^3&=(0, \tau_1^3, \pi),\\
			\tau^4&=(0, \tau_1^4, \tau_2^4, \pi),
		\end{split}\]
		and where $\tau_1^d$ and $\tau_2^d$ are any numbers satisfying $l \leq \tau_1^d \leq  lp/2$ and $l(p+2)/2 \leq \tau_2^d \leq pl$. 
	\end{theorem}
	
	\begin{proof}
		In the proof we omit the notation $\tau^d$ and instead write $\tau$, but note that $\tau$ still depends on $l$ and thus on $d$ since $l=4 \pi /(dp+2)$. By Theorem \ref{thm:cor61analogi}, letting $\Lambda_d$ be the set of sequences that may serve as the set of arguments of zeroes of the extremal function, it follows that 
		\[\mathscr{C}_{d,p} \leq \sup_{\tau \in \Lambda_d} \frac{1}{\pi} E_{d,p}(\tau).\]
		In particular, due to Lemma \ref{lem:Turan} and Lemma \ref{lem:d4pi2} we let $\Lambda_d$ be as defined in \eqref{eq:Lambdad} and \eqref{eq:Lambda4} respectively. Assume $\tau$ is a sequence in $\Lambda_d$ such that $E_{d,p}(\lambda) \leq E_{d,p}(\tau)$ for all $\lambda$ in $\Lambda_d$. By Lemma \ref{lem:tleq}, we may assume $\tau_n \leq npl/2$. Furthermore letting $M_1=l/2$ be the midpoint between $(0, l)$ and its intersecting interval at level $n+1$ we have $M_1 \geq \pi/d$, and so by Lemma \ref{lem:mid} we must have $\tau_1 \geq l$. This concludes the proof for $d=2$ and $d=3$ since $\tau_1$ is the only variable determining $E_{d,p}(\tau)$ in this case, and clearly $E_{d,p}(\tau)$ obtains the same value for any $l \leq \tau_1 \leq lp/2$. 
		
		Now consider $d=4$. From above we know that $l \leq \tau_1 \leq lp/2$. By assumption $\tau_2 \geq \pi/2$, and by Lemma \ref{lem:tleq} we have $\tau_2 \leq pl$. Let $M_2=(3p-2)\pi/(4p+2)$ be the midpoint between the second interval at level $1$ and the overlapping interval at level $2$. If $2 \leq p \leq 3$ it follows that $\tau_2 \geq \pi/2 \geq M_2$ and so by Lemma \ref{lem:mid} it follows that $\tau_2 \geq l(p+2)/2$. For $3 \leq p \leq 4$ it follows from Lemma \ref{lem:d4p34} that $ \tau_2 \geq l(p+2)/2$. That is $l \leq \tau_1 \leq lp/2$ and $l(p+2)/2 \leq \tau_2 \leq pl$, and all such sequences give rise to the same value for $E_p(\tau)$, which concludes the proof.
	\end{proof}
	
	Due to a numerical analysis we expect the zeroes of the extremal function $\varphi_{d, p}$ to be within the set of sequences $\tau$ given in Theorem \ref{thm:uppergood}. In particular, this means that further information on the zeroes of the extremal function $\varphi_{d, p}$ will not improve on the upper bound for $\mathscr{C}_{d,p}$ from Theorem \ref{thm:uppergood}, given that we use the same technique. Finally, we conclude the proof of Theorem \ref{thm:main} by showing that the upper bounds on $\mathscr{C}_{d,p}$ from Theorem \ref{thm:uppergood} are in fact better than $dp/2+1$. 
	
	\begin{proof}[Proof of Theorem \ref{thm:main}]
		Let $2 \leq d \leq 4$ be an integer and let $2 \leq p \leq 4$. The upper bounds from Theorem \ref{thm:uppergood} can be expressed as follows:
		\[\begin{split}
			\mathscr{C}_{2,p} &\leq \frac{1}{\pi} \int_{0}^{l} \left(\frac{\sin ( \frac{p+1}{2} \theta)}{\sin \frac{\theta}{2}} \right)^2 d\theta+ \frac{1}{\pi}\int_{lp/2}^{\pi} \left(\frac{\sin ( \frac{p+1}{2} \theta-\frac{p}{2}\pi)}{\sin \frac{\theta}{2}} \right)^2 d\theta\\
			\mathscr{C}_{3,p} & \leq
			\frac{1}{\pi}\int_{0}^{l} \left(\frac{\sin ( \frac{3p+2}{4} \theta)}{\sin \frac{\theta}{2}} \right)^2 d\theta+ \frac{1}{\pi}\int_{lp/2}^{l(p/2+1)} \left(\frac{\sin ( \frac{3p+2}{4} \theta-\frac{p}{2}\pi)}{\sin \frac{\theta}{2}} \right)^2 d\theta\\
			\mathscr{C}_{4,p} & \leq  \frac{1}{\pi}\int_{0}^{l} \left(\frac{\sin ( \frac{2p+1}{2} \theta)}{\sin \frac{\theta}{2}} \right)^2 d\theta+ \frac{1}{\pi}\int_{lp/2}^{l(p/2+1)} \left(\frac{\sin ( \frac{2p+1}{2} \theta-\frac{p}{2}\pi)}{\sin \frac{\theta}{2}} \right)^2 d\theta\\&+ \frac{1}{\pi}\int_{lp}^{\pi} \left(\frac{\sin ( \frac{2p+1}{2} \theta-p\pi)}{\sin \frac{\theta}{2}} \right)^2 d\theta.
		\end{split}\]
		Using that $\sin x/2$ is increasing for $0<x<\pi$ it follows that these upper bounds are all less than or equal to $A_{d,p}$ where 
		\[A_{d,p}=\frac{1}{\pi} \int_{0}^{2\pi(d+1)/(dp+2)} \left( \frac{\sin \big( \frac{dp+2}{4} \theta \big)}{\sin \frac{\theta}{2}}\right) \, d\theta.\]
		By a substitution, we see that the inequality $A_{d,p} \leq dp/2+1$ is equivalent to 
		\begin{equation}\label{eq:p2} \frac{2}{\pi} \frac{4}{(dp+2)^2} \int_{0}^{\pi(d+1)/2} \left( \frac{\sin \theta}{\sin \frac{2}{dp+2} \theta }\right)^2 \, d\theta \leq 1.
		\end{equation} A straightforward computation shows that \eqref{eq:p2} is attained for $p=2$. Thus we conclude that \eqref{eq:p2} also holds for all $2<p \leq 4$ by the pointwise estimate
		\[\frac{4}{(dp+2)^2}\left( \frac{\sin \theta}{\sin \frac{2}{dp+2} \theta }\right)^2 \leq \frac{4}{(2d+2)^2}\left( \frac{\sin \theta}{\sin \frac{2}{2d+2} \theta }\right)^2. \]
		This concludes the proof. 
	\end{proof}
	We do not move beyond $d > 4$ since we do not have sufficient information about the arguments of the zeroes of the extremal function when $d \geq 5$. In particular, we cannot invoke Lemma \ref{lem:mid} for all elements as we cannot justify assuming that the arguments are bigger than the related midpoints between intervals at level $n$ and $n+1$. 
	Note that one can show that our technique is \emph{not} good enough to verify the conjecture $\mathscr{C}_{d,p} \leq dp/2+1$ for large $d$. For example, for $p=5/2$ and $d=6$ the supremum \eqref{eq:1546} with the restrictions $\tau_2 \geq \tau_1 \geq \pi/d$ and $\tau_3 \geq \pi/2$ is at least $8.6>dp/2+1=8.5$. With more information about the zeroes this could be improved. 
	
	We also note that if for $2 \leq p \leq 4$ and any $d \geq 2$ we knew that the arguments of the zeroes of the extremal function satisfy $t_j \geq (2pj-p-1)\pi/(dp+2)$, which is the midpoint between an interval at level $n-1$ and $n$, then we could generalize Theorem \ref{thm:uppergood}, and in particular prove the conjectured upper bound $\mathscr{C}_{d,p} \leq dp/2+1$. We expect the zeroes of the extremal function to satisfy this condition, since we expect that $t_j$ is somewhat close to $\gamma_j$, where $\gamma_j=(2-p+2pj)\pi/(dp+2)$. 
	
	\section{Proof of Theorem \ref{thm:plarge}}\label{sec:plarge}
	In this section we prove Theorem \ref{thm:plarge}, which gives an upper bound for $\mathscr{C}_{d,p}$ when $p$ is large. We use a technique similar to that in \cite[Section 7]{Brevig}. First, however, we need the following lemma by Hylt\'en-Cavallius \cite[Theorem 1]{Cavallius}.

	\begin{lemma}\label{lem:Cavllius}
		Let $P(e^{i\theta})$ be a polynomial of degree $d \geq 1$ which is maximized for $\theta=0$ and where $\|P\|_\infty=P(1)=1$. Then
		\[|P(e^{i \theta})| \geq \|P\|_\infty \cos( d \theta/2),\] for any $|\theta| \leq \pi/d$. 
	\end{lemma}
	
	\begin{proof}[Proof of Theorem \ref{thm:plarge}]
		Let $P(e^{i\theta})$ be a polynomial of degree $d$. By rotating we may assume the function is maximized for $\theta=0$, where $\|P\|_\infty=P(1)=1$. Then from Lemma \ref{lem:Cavllius}
		it follows that 
		\[|P(e^{i \theta})| \geq \|P\|_\infty \cos( d \theta/2),\] for any $|\theta| \leq \pi/d$. Raising both sides to the power $p$ and integrating yields that
		\[ \int_{-\pi/d}^{\pi/d} |P(e^{i\theta})|^p\, d \theta \geq \|P\|_\infty^p \int_{-\pi/d}^{\pi/d} (\cos(d \theta/2))^p \, d\theta=\frac{2}{d} B((p+1)/2,1/2).\] 
		Hence it follows that 
		\[\|P\|_p^p \geq \frac{B((p+1)/2,1/2)}{\pi d}, \]
		and so by \eqref{eq:extprob} it follows that $\mathscr{C}_{d,p} \leq d \pi/B((p+1)/2,1/2)$.
	\end{proof}
	
	Due to well-known asymptotics for the beta function it follows from Theorem \ref{thm:plarge} that as $p \to \infty$ we have 
	\[\mathscr{C}_{d,p} \leq d\left( \sqrt{\frac{\pi p}{2}}+O\left(\frac{1}{\sqrt{p}}\right) \right).\]
	Kershaw \cite{Kershaw} has shown that if $0<s<1$ then
	\begin{equation}\label{eq:Kershaw}
		\left(x+\frac{s}{2}\right)^{1-s}<\frac{\Gamma(x+1)}{\Gamma(x+2)},
	\end{equation}
	for all $x > 0$.
	Using \eqref{eq:Kershaw} with $s=1/2$ we see that
	\[B(x,1/2)=\frac{\Gamma(x)\Gamma(1/2)}{\Gamma(x+1/2)}=\frac{\sqrt{\pi} \Gamma(x+1)}{x \Gamma(x+1/2)} > \frac{\sqrt{\pi(x+1/4)}}{x},\]
	for any $x>0$. 
	Consequently, it follows that 
	\[\mathscr{C}_{d,p}<\frac{d  \sqrt{\pi}(p+1)}{\sqrt{2p+3}},\]
	and in particular if $p>6.8$ then $\mathscr{C}_{d,p} \leq dp/2$.
	\begin{figure}
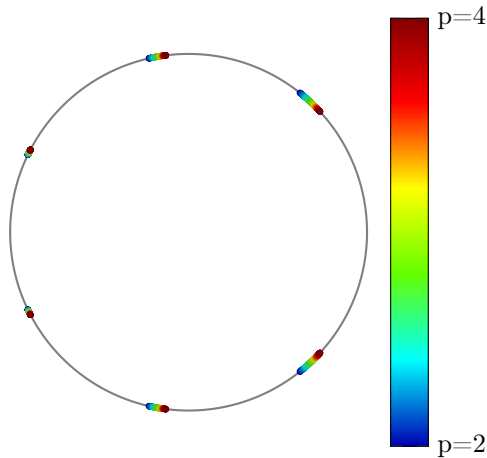

		\begin{tikzpicture}
			\begin{axis}[
				axis equal image,    
				axis lines=none,     
				ticks=none,             
				colormap/bluered,
				point meta min=0,
				point meta max=1,
				colorbar,
				colorbar style={
					ytick={0,1},          
					yticklabels={p=2,p=4} 
				},
				]
				\input{data_plot_1.tex}
				\input{data_plot_2.tex}
				\input{data_plot_3.tex}
				\input{data_plot_4.tex}
				\input{data_plot_5.tex}
				\input{data_plot_6.tex}
				\addplot[
				domain=0:2*pi,
				samples=200,
				thick,
				gray
				]
				({cos(deg(x))}, {sin(deg(x))});
				
			\end{axis}
		\end{tikzpicture}
		\caption{The expected zeroes of the extremal function $\varphi_{d,p}$ for $d=6$ and $2 \leq p \leq 4$ based on numerical optimization.}
		\label{fig:d6}
	\end{figure}
	\section{Numerical estimates}\label{sec:num}
	
	Consider the polynomial $Q$ of degree $d$ with the zeroes $e^{\pm \gamma_j i}$ where the argument is defined by $\gamma_j=\pi(2-p+2pj)/(dp+2)$, for $1 \leq j \leq \lfloor d/2 \rfloor$, and if $d$ is odd the last zero is $z=-1$. Then $1/\|Q\|_p^p$ provides a good lower bound for $\mathscr{C}_{d,p}$. In particular we have the following
	
	\begin{theorem}\label{thm:plowergen}
		Let $d \geq 1$ be an integer and let $\gamma_j=(2-p+2pj)\pi/(dp+2)$. Then  $\mathscr{C}_{d,p} \geq 1/C$, where
		\[C=\frac{1}{\pi}\int_{0}^{\pi} \prod_{j=1}^{d/2} \left|\frac{\cos x - \cos \gamma_j}{1-\cos \gamma_j} \right|^p \, d\theta,\]
		if $d$ is even and
		\[C=\frac{1}{\pi}\int_{0}^{\pi} |\cos x/2|^p \prod_{j=1}^{d/2} \left|\frac{\cos x - \cos \gamma_j}{1-\cos \gamma_j}\right|^p \, d\theta,\]
		if $d$ is odd.
	\end{theorem}	
	
	Observe that for $p=2$ the polynomial $Q$ is the extremal function for \eqref{eq:extprob}, and the lower bound is equal to $\mathscr{C}_{d,2}=d+1$. We also note that for $2 \leq p \leq 4$ the sequence $\gamma$ is one of the sequences attaining the supremum $\sup_{\tau \in \Lambda_d} E_{d,p}(\tau)$ and providing the upper bound for $\mathscr{C}_{d,p}$ in Theorem \ref{thm:uppergood}. Clearly one can obtain a slightly better lower bound for $\mathscr{C}_{d,p}$ by numerically optimizing over all polynomials. The true bound is still close to the bound stated in Theorem \ref{thm:plowergen}.
	
	Based on numerical computations we also expect the arguments of the zeroes of the extremal function to be somewhat close to $\gamma_j$. In particular we expect the $j$th argument to be greater than $ (2pj-p-1)\pi/(dp+2)$, which as mentioned in Section \ref{sec:mainproof} is enough to show the conjecture $\mathscr{C}_{d,p} \leq dp/2+1$. Figure \ref{fig:d6} illustrates the expected zeroes of the extremal function $\varphi_{6,p}$ for $2 \leq p \leq 4$ based on numerical optimization. We note that the numerical estimates support both the conjecture $\mathscr{C}_{d,p} \leq dp/2+1$, and that $\mathscr{C}_{d,p}/(dp/2+1)$ is decreasing in $p$. 
	\bibliographystyle{amsplain} 
	\bibliography{poly} 
	\appendix
\end{document}